\documentclass[12pt,reqno]{amsart}
\usepackage{amsmath}
\usepackage{amssymb}  % defines \nmid, for example
\usepackage{latexsym} % for \Box
\usepackage{comment}
\usepackage{url}
\usepackage{fullpage,url,amssymb,amsmath,amsthm,amsfonts,mathrsfs}
\usepackage[usenames,dvipsnames]{color}
\usepackage[pagebackref = true, colorlinks = true, linkcolor = blue, citecolor = Green]{hyperref}
\usepackage[lite]{amsrefs}
\usepackage{enumitem}
\usepackage{amscd}   % for commutative diagrams
\usepackage[all, cmtip]{xy} % for complicated commutative diagrams
\usepackage{xfrac}
\usepackage[T1]{fontenc}

\DeclareFontEncoding{OT2}{}{} % to enable usage of cyrillic fonts

\usepackage[all]{xy}
\usepackage{fullpage}

\usepackage{color} 
\newcommand{\todo}[1]{{\color{blue} \sf TO DO}}

\def\act#1#2%
  {\mathop{}%
   \mathopen{\vphantom{#2}}^{#1}%
   \kern-3\scriptspace%
   #2}

\newcommand{\Q}{{\mathbb Q}}

\newcommand{\C}{{\mathbb C}}
\newcommand{\F}{{\mathbb F}}

\DeclareMathOperator{\Jac}{Jac}

\newtheorem{Theorem}{Theorem}[section]
\newtheorem{Lemma}[Theorem]{Lemma}

\newtheorem{Corollary}[Theorem]{Corollary}

\newtheorem{Remark}[Theorem]{Remark}
\newtheorem{Conjecture}[Theorem]{Conjecture}

\theoremstyle{definition}

\numberwithin{equation}{section}

\begin{document}
\title{Recovering curves from isogeny classes of Prym varieties}

\author{Jos\'e Felipe Voloch}
\address{School of Mathematics and Statistics, University of Canterbury, Private Bag 4800, Christchurch 8140, New Zealand}
\email{felipe.voloch@canterbury.ac.nz}
\urladdr{http://www.math.canterbury.ac.nz/\~{}f.voloch}

\begin{abstract}
We propose a conjecture on how to recover an algebraic curve from the
isogeny classes of its Prym varieties. Motivated by this conjecture, we analyze constructions of pairs of curves with isogenous Jacobians and several isogenous Pryms.
\end{abstract}

\dedicatory{To the memory of Yuri Zarhin}

\maketitle

\section{Introduction}
\label{sec:intro}

For our purposes, a curve is a smooth, irreducible, projective algebraic curve defined over some field and an abelian variety is a smooth, irreducible, projective algebraic group variety defined over some field. An isogeny between abelian varieties of the same dimension is an algebraic group homomorphism of finite degree. Beyond the introduction, we will be interested in varieties defined over the complex numbers or families of such.

To a curve, one can associate canonically an abelian variety called its Jacobian. The Jacobian comes with some additional structure, namely a principal polarization, and Torelli's theorem states that one can recover the curve from its Jacobian together with its principal polarization. However, if we ignore the principal polarization, the Jacobian alone is not enough to recover the curve, i.e., there are non-isomorphic curves with isogenous Jacobians (viz. the classical Richelot isogeny) and even isomorphic Jacobians (see e.g. \cite{Howe}).

In addition to a given curve, we can consider its unramified covers and their associated Jacobians. These decompose as a product of the Jacobian of the original curve and the so-called Prym varieties. In \cite[Conjecture 2.2]{SV}, a conjecture was raised on how to recover a curve of genus at least two defined over a finite field from the Jacobians of its Hilbert class field covers up to isogeny. In the case of finite fields, the isogeny class of an abelian variety can be characterized by its zeta function by Tate's isogeny theorem and the conjecture was stated in terms of zeta functions. Additionally, the zeta function can be computed efficiently, so the conjecture was extensively tested. A variant of this conjecture, using double covers, was asked as a question on the same paper (\cite[Question 2.3]{SV}). A third variant of the conjecture, still in the context of finite fields, was then proved in \cite{BV}.

The main purpose of this paper is to formulate a conjecture that encompasses the conjectures mentioned above and discuss a few variants. Motivated by this conjecture, we also analyze constructions of non-isomorphic pairs of curves with isogenous Jacobians and several isogenous Pryms so as to provide a lower bound on the number of necessary Pryms needed to obtain the isomorphism between the curves.

Yuri Zarhin was particularly interested in and made substantial contributions to all aspects of the theory of abelian varieties. Closer to the topic of this paper, he wrote a number of papers (e.g. \cites{Z1,Z2}) with broad conditions for Jacobians not to be isogenous.

\section{Main conjecture}

We denote by $\pi_1(C)$, as usual, the fundamental group of a complex curve and recall that a surjection $\pi_1(C) \to G$ to a finite group $G$ defines an \'etale cover $C_1 \to C$ of curves with Galois group $G$. The \'etale cover induces a norm map of abelian varieties $\Jac(C_1) \to \Jac(C)$ (where $\Jac(C_1), \Jac(C)$ are the Jacobians of $C_1,C$ respectively) and the connected component of the identity of the kernel of this norm map is the Prym variety associated to the cover.

\begin{Conjecture}
\label{conj:main}
Let $C$ and $C'$ be complex curves of genus at least two and $\varphi:\pi_1(C)\to \pi_1(C')$ an isomorphism.
Assume that, for every finite quotient $\pi_1(C') \to G$, the corresponding covers $C_1$ and $C'_1$ are such that there is an isogeny $\Jac(C_1) \to \Jac(C'_1)$ commuting with the action of $G$. Then $\varphi$ comes from an isomorphism $C \to C'$ up to an inner automorphism of $\pi_1(C)$, accounting for base points.
\end{Conjecture}

Equivalently, the map $f \mapsto f_*$ from $\{f:C\to C'\}$ to $\{\psi:\pi_1(C)\to \pi_1(C')\}/\operatorname{Inn}(\pi_1(C'))$ should induce a bijection to the image of the subset of the $\psi's$ that satisfy the condition of the conjecture.

A different conjecture in a similar spirit has been proposed by Prasad and Rajan \cite{01902899}.

One could consider a weakening of the conjecture in which the matching of the Pryms does not come from a map $\varphi$ as above. One can also consider strengthenings of the conjecture by considering instead isomorphisms $\varphi$ between the abelianization of the fundamental groups or the exponent-two quotients of this abelianization. The latter is false for genus two (see Section \ref{sec:extra}) but still plausible for higher genus. Extensive calculations for small genus and small finite fields (see \cites{SV, doubly}) justify this. Any version of the conjecture that uses only finitely many groups $G$ for a fixed genus $g$ would follow in the general case from the statement over finite fields by specialization. The results of \cite{BV} are not strong enough, alas.

We now sketch how the hypotheses of Conjecture \ref{conj:main} imply the hypotheses of \cite[Theorem 2.5]{BV} and thus the conclusion of Conjecture \ref{conj:main} for curves over finite fields (up to Frobenius twists). 

For an abelian variety $A/\mathbb{F}_q$, let $F$ be its Frobenius isogeny. Then $ \ker(1-F^n)=A(\mathbb{F}_{q^n})$. So, for $A=\Jac(C)$, the isogeny $1-F^n$
produces an \'etale abelian cover of $C$ with Galois group $G=\Jac(C)(\F_{q^n})$. An isogeny of the Jacobians of such covers implies equality of the corresponding zeta functions. Now, the zeta function of the cover is the product of $L$-functions and, if this isogeny commutes with the action of $G$, then the equality of zeta functions implies the equality of the $L$-functions, as required to apply \cite[Theorem 2.5]{BV}.

\section{Families}

The paper \cite{MNP} proves that the Jacobian of the generic point of a dimension $3g-3-k$ family of curves of genus $g$ is not isogenous to a different Jacobian, if $g > 3k+4$. Hence we can distinguish curves using their Jacobians up to isogeny, without needing coverings, for sufficiently general curves. In a similar direction, \cite{LM} proves a comparable result for Pryms of double covers. On the other hand, the theorem below shows that there are two distinct $g+1$-dimensional families of pairs of curves of genus $g$ with isogenous Jacobians. We will extend this construction to include Pryms as well.

\begin{Theorem}
\label{thm:mestre}
(Mestre, \cite{Mestre}) The hyperelliptic curves $C, C'$ below, where $v, a_1, \ldots, a_g$ are variables, $b_i=(a_i v^2-1)/(a_i-v^2), 1 \leq i \leq g$, have isogenous Jacobians.

\begin{align*}
C: y^2=& \,(x-v)(v x-1)\left(x^2-a_1\right) \ldots\left(x^2-a_g\right)\\
C': y^2=& \,(x-v)\left(v x-(-1)^g\right)\left(x^2-b_1\right) \ldots\left(x^2-b_g\right)
\end{align*}
%C needs to be twisted by $A=2\left(v^2+1\right) \prod_{i=1}^g\left(v^2-a_i\right)$ to work over any field
\end{Theorem}

\begin{Remark}
In the original statement in \cite{Mestre}, $C$ is twisted by a factor $A$ so that the statement works with $\Q$ as the field of constants. As we are primarily interested in $\C$ as the field of constants, we omit the twist.
\end{Remark} 

Recall that the \'etale double covers of a hyperelliptic curve of genus $g$ given by $y^2=f(x), \deg f = 2g+2$, are obtained by
factoring $f(x)=k(x)h(x)$ with $k,h$ coprime and $\deg k, \deg h$ both even and taking the cover $z^2 = k(x)$. The corresponding Prym variety is isogenous to the product of the Jacobians of $y^2=k(x)$ and $y^2=h(x)$. We say that the curve $y^2=h(x)$ is complementary to $y^2=k(x)$.

\begin{Lemma}
\label{lem:gpt}
Notation as in Theorem \ref{thm:mestre}.
Let $T(a)=(v^2a-1)/(a-v^2),R(a)=1/a$. These are commuting involutions.
Assume $S \subset \{a_1,\ldots,a_g\}, \#S = 2e >0$. If either 
$T(S)=S$ or $R(T(S))=S$, then the double covers of $C,C'$ respectively associated with 
$$k_S(x)=\prod_{a\in S}(x^2-a), k'_S(x)=\prod_{a\in S}(x^2-T(a))$$
 have isogenous Pryms.
\end{Lemma}

\begin{proof}
One immediately checks that $R,T$ are commuting involutions.
For a given set $S$ and polynomials $k_S,k'_S$, the complementary curves form a pair as in Theorem \ref{thm:mestre} of genus
$g -2e$, so their Jacobians are isogenous by that theorem. Moreover,
if $T(S)=S$, then $k'_S=k_S$ so the curves $y^2=k_S,y^2=k'_S$ are equal and so, in particular have isogenous Jacobians. If $R(T(S))=S$, then $T(S)=R(S)$
and the curves $y^2=k_S,y^2=k'_S$ are isomorphic using $x \mapsto 1/x$ and rescaling $y$ and again have isogenous Jacobians.
\end{proof}

To apply the lemma, assume $g>4$ and let $r=[g/4], 4 \nmid g, r= g/4-1, 4|g$. Let $a_{4j+1},j=0,\ldots,r-1$ and $a_i, i> 4r$ be
independent variables and define 
$$a_{4j+2}=T(a_{4j+1}),a_{4j+3}=R(a_{4j+1}), a_{4j+4}=R(T(a_{4j+1})), j=0,\ldots,r-1.$$ 
Now, we can take for $S$ any non-empty subset consisting 
of a union of pairs $\{a_{2j+1},a_{2j+2}\}, j \le 2r-1$ and it will satify $T(S)=S$ or we can take for $S$ any non-empty subset consisting of a union of pairs $\{a_{4j+1},a_{4j+4}\}$ or $\{a_{4j+2},a_{4j+3}\}$, $j \le r$ and
it will satify $RT(S)=S$. The sets consisting of a union of quadruples $\{a_{4j+1},a_{4j+2},a_{4j+3},a_{4j+4}\}$ occur on both counts. So we get at least $2\cdot 4^r - 2^r -1$ suitable 
sets $S$, leading to this many isogenous Pryms (out of a total of $4^g-1$ Pryms) in the generic case of a family of dimension $r+g-4r=g-3r$. To summarize:

\begin{Corollary}
For $g>4$, let $r=[g/4], 4 \nmid g, r= g/4-1, 4|g$.
There is a specialization of the family of pairs of curves of Theorem \ref{thm:mestre} to a
family of dimension $g-3r = g/4 +O(1)$ which, in addition to having isogenous Jacobians, have $2\cdot 4^r - 2^r -1$ isogenous Pryms.
\end{Corollary}

 It is possible to specialize further and eke out a few more isogenous Pryms on a family of lower dimension.

\medskip

Another way of constructing families of pairs of curves with isogenous Jacobians is due to Smith \cite{Smith}, which can also be upgraded to include isogeny between some of the Pryms. Smith's construction is as follows:

\begin{Theorem}
\label{thm:smith}
(Smith, \cite{Smith}) If the hyperelliptic curves $C_1, C_2$ are defined by $y_i^2=f_i(x_i), i=1,2$ and $f_1(x_1)-f_2(x_2)$ factors as $A(x_1,x_2)B(x_1,x_2)$ then there is a
homomorphism between the Jacobians of $C_1, C_2$ induced by the correspondence
$X \subset C_1 \times C_2$ defined by $A(x_1,x_2)=y_1-y_2=0$.
\end{Theorem}

The families thus obtained are of smaller dimension than those of Theorem \ref{thm:mestre} for the same genus but they
can be arranged to have some isogenous Pryms without imposing additional equations, provided it can be proved that the homomorphisms are isogenies, as follows.

Consider the hyperelliptic curves $C_1', C_2'$ defined by
$y_i^2=F(f_i(x_i)), i=1,2$, where $F$ is some polynomial, then we have

$$F(f_1(x_1))-F(f_2(x_2)) = A(x_1,x_2)B(x_1,x_2)\left(\frac{F(f_1(x_1))-F(f_2(x_2))}{f_1(x_1)-f_2(x_2)}\right).$$

So there is a homomorphism between the Jacobians of $C_1', C_2'$ and, for each factorization $F=HK$, there is a homomorphism between the Pryms obtained
from the factorizations $F(f_i) = H(f_i)K(f_i)$ (using the same identities with $F$ replaced by $K,H$, assuming $\deg f_i$ is even or $\deg H, \deg K$ both even). One still needs to impose some conditions to ensure these homomorphisms are isogenies.

As this method, even when successful, leads to fewer isogenous Pryms than the previous method, we do not pursue it further.

Yet another way of constructing isogenous Jacobians with at least one isogenous Prym is via Sunada's method using Gassmann triples (see e.g. \cite{06713697}). The isogenous Prym comes for free as the curves of the relevant pair are both quotients of the same curve.

\section{The extraordinary curves} 
\label{sec:extra}

For $r\ne \pm 27, r\neq 23 \pm 10\sqrt{-2}$, let $C_{r}$ denote the curve given by
\begin{equation}
\label{eqn:basiccurve}
C_r: y^2 = x^6 + (r-18)x^4 + (81-2r)x^2 + r.
\end{equation}

The following result is proved in \cite{doubly} and provides an example of a pair curves of genus $2$ with isogenous Jacobians and isogenous Pryms for {\em all} double covers (and a few triple covers as well).
This is an extraordinary coincidence, hence the name ``extraordinary curves.'' In \cite{doubly} we place the existence of this pair in the context of the Zilber--Pink conjecture and prove that this conjecture implies that there can be only finitely many such pairs in the family (\ref{eqn:basiccurve}). We expect that, in fact, it is the only such pair and the calculations reported in \cite{doubly} strongly suggest that.

\begin{Theorem}
\label{thm:extraordinary} 
Let $r_1$ and $r_2$ be the roots of $x^2 - 27x + 1$. Let $C_{r_1}$ and $C_{r_2}$, as in equation (\ref{eqn:basiccurve}).
Then $C_{r_1}$ and $C_{r_2}$ have isogenous Jacobians and isogenous matching Pryms for {\em all} double covers for a suitable map $\varphi:\pi_1(C_{r_1})^{ab}/2\to \pi_1(C_{r_2})^{ab}/2$ (the quotients of the abelianized fundamental groups by squares).

% Then the Weierstrass points of $C_{r_1}$ and $C_{r_2}$ are rational over~$L$, and $C_{r_1}$ and $C_{r_2}$ are doubly isogenous over $L$ when we embed them into their Jacobians using a Weierstrass point as the base point. Furthermore, the Pryms of the intrinsic triple covers of $C_{r_1}$ and $C_{r_2}$ are also isogenous over~$L$.
\end{Theorem}

\begin{Remark}
In addition to the matching double cover Pryms, there are two triple covers of $C_{r_1}$ and $C_{r_2}$ (dubbed intrinsic triple covers in \cite{doubly}) that also have isogenous Pryms.
\end{Remark}

\section*{Acknowledgements and statement of AI use}
This work was supported by the Marsden Fund administered by the Royal Society of New Zealand. I would like to thank H. Esnault and C. Voisin for helpful discussions.

A substantial portion of this work was carried out before AI became useful for mathematical research. I spoke about Conjecture \ref{conj:main} in the complex case at a conference at Brown University in 2023. Recently, I used ChatGPT to convert my handwritten notes from that talk to LaTeX. The use of Mestre's Theorem \ref{thm:mestre} to obtain isogenous Pryms was already suggested in \cite{SV} and carried out further in an early draft of this work. I asked ChatGPT to review that draft and it suggested Lemma \ref{lem:gpt} which greatly improved on my earlier construction. I take responsibility for all the proofs and the presentation.

%%%%%%%%%%%%%%%%%%%%%%%%%%%%%%%%%%%%%%%%%%%%%%%%%%%%%%%%%%%%%%%%%%%%%%%%%%%%%%%%
%%%%%%%%%%%%%%%%%%%%%%%%%%%%%%%%%%%%%%%%%%%%%%%%%%%%%%%%%%%%%%%%%%%%%%%%%%%%%%%%
%%%%%%%%%%%%%%%%%%%%%%%%%%%%%%%%%%%%%%%%%%%%%%%%%%%%%%%%%%%%%%%%%%%%%%%%%%%%%%%%
%%%%%%%%%%%%%%%%%%%%%%%%%%%%%%%%%%%%%%%%%%%%%%%%%%%%%%%%%%%%%%%%%%%%%%%%%%%%%%%%
%%%%%%%%%%%%%%%%%%%%%%%%%%%%%%%%%%%%%%%%%%%%%%%%%%%%%%%%%%%%%%%%%%%%%%%%%%%%%%%%
	
\section*{References}
%\begin{bibdiv}
\bibliographystyle{plain}

%\end{bibdiv}
	
\end{document}